\documentclass[a4paper,10pt, reqno, english]{amsart}

\usepackage[table]{xcolor}
\usepackage[utf8]{inputenc}
\usepackage{url,paralist}
\usepackage{caption}
\usepackage{subcaption}
\usepackage{makecell} 
\usepackage{color}
\usepackage{float}
\usepackage{dsfont}
\usepackage{tikz}
\usepackage[toc,page]{appendix}
\usepackage{amsfonts,amssymb,enumerate}
\usepackage{graphicx}
\usepackage{amsmath,bm,amsthm}
\usepackage{mathtools, mathrsfs}
\usetikzlibrary{arrows.meta}
\usepackage[colorlinks=true,urlcolor=blue,citecolor=magenta]{hyperref}
\usepackage[arrow,curve,matrix,tips,2cell]{xy}  
\usepackage{xcolor,colortbl}

\theoremstyle{plain}
\newtheorem{theorem}{Theorem}[section]
	\newtheorem{innercustomthm}{Theorem} 
	\newenvironment{customthm}[1]
	  {\renewcommand\theinnercustomthm{#1}\innercustomthm}
	  {\endinnercustomthm}
\newtheorem{lemma}[theorem]{Lemma}
\newtheorem{proposition}[theorem]{Proposition}
\newtheorem{corollary}[theorem]{Corollary}

\theoremstyle{definition}
\newtheorem{definition}[theorem]{Definition}

\newtheorem{remark}[theorem]{Remark}

 \usepackage{setspace}

{\catcode`@=11\global\let\c@equation=\c@theorem}

\newcommand{\calS}{{\mathcal S}}

\newcommand{\IR}{{\mathbb R}} 
\newcommand{\IC}{{\mathbb C}} 
\newcommand{\IH}{{\mathbb H}} 
\newcommand{\IO}{{\mathbb O}} 
\newcommand{\ImO}{{\rm{ Im }\, \IO}} 
\newcommand{\IF}{{\mathbb F}} 
\newcommand{\IZ}{{\mathbb Z}} 

\newcommand{\ind}{\operatorname{Index}} 
\newcommand{\id}{\operatorname{id}} 
\newcommand{\CP}[1]{\mathbb{CP}^{#1}} 
\newcommand{\Gr}[2]{ G_{#1}( \mathbb{R}^{#2}) } 
\newcommand{\OGr}[2]{ \widetilde{G}_{#1}( \mathbb{R}^{#2}) } 
\newcommand{\CBorel}[2]{ E#1 \times_{#1} #2} 
\newcommand{\Index}[4]{\ind_{#3}^{#2}(#1;#4) } 

\newcommand{\Map}[3]{ #1 \colon #2 \longrightarrow #3  } 

\newcommand{\bigC}[1]{ \big(  #1 \big) } 
\newcommand{\BigC}[1]{ \Big(  #1 \Big) } 
\newcommand{\free}[1]{\langle #1 \rangle} 

\newcommand{\xySquare}[9] {	 
	\xymatrix@C=#9mm{
		#1 \ar[r]^{#2} \ar[d]_{#4} & #3 \ar[d]^{#5}  \\
		#6\ar[r]^{#7} & #8 }
}
\newcommand{\xyTriangle}[6] {	 
	\xymatrix{
		#1 \ar[rr]^{#2} \ar[dr]_{#4} & & #3 \ar[dl]^{#5}  \\
		& #6 & }
}
\newcommand{\FiberBundle}[4] {	 
	\xymatrix{
		#1 \; \ar@{^{(}->}[r] & #2 \ar[d]^{#3}  \\
		& #4 }
}

\newcommand{\version}[1]{ 
	\begin{center} 
		last edited on #1\\
		last compiled on \today\\
		name of texfile: \jobname
	\end{center}
}

\title[The index of fibrations with total space a $G_{2}$ flag manifold]
{The ideal-valued index of fibrations with total space a $G_{2}$ flag manifold}

\author{No\'{e} B\'{a}rcenas}
\address{
	\hfill\break Centro de Ciencias Matem\'aticas. UNAM, Ap. Postal 61-3 Xangari \\ 
	\hfill\break Morelia, Michoac\'an, México C.P. 58089 
}
\email{barcenas@matmor.unam.mx}
\author{Jaime Calles Loperena}
\address{
	\hfill\break Instituto de Matem\'aticas, Universidad Nacional Aut\'onoma de M\'exico,  C.U.  \\
	\hfill\break CDMX 04510, Mexico.
}
\email{calles@im.unam.mx}

\begin{document}


\date{\today}
\maketitle


\begin{abstract} 

Using the cohomology of the $G_2$-flag manifolds $G_2/U(2)_{\pm}$, and their structure as a fiber bundle over the homogeneous space $G_2/SO(4)$, we compute the $\IZ_2$ Fadell-Husseini index of such fiber bundles for the 
$\IZ_2$ action given by complex conjugation.

\medskip
Also, considering the tautological bundle $\gamma$ over $\OGr{4}{7}$, we compute the $\IZ_2$ Fadell-Husseini index of the pullback bundle of $s\gamma$ along the composition of the fiber bundle $  G_2/U(2)_{\pm} \to G_2/SO(4)$, the embedding between $G_2/SO(4)$ and $\OGr{3}{7}$, and the map that takes the orthogonal complement of a subspace.
Here $s\gamma$ means the associated sphere bundle of $\gamma$. Furthermore, we derive a general formula for the $n$-fold product bundle $s\gamma^n$ for which we make the same computations. 
We finish our work with an application of our computations in a problem concerning discrete geometry.



\end{abstract}

\section{ \Large Introduction.}
\label{Introduction}

A \emph{generalized flag manifold} is an homogeneous space of the form $G/C(T)$, where $G$ is a semisimple, compact and connected Lie group, and $C(T)$ is the centralizer of a torus $T \subset G$. In case that $T$ is the maximal torus, then $T = C(T)$ and we call $G/T$ a \emph{complete flag manifold}.

The group in which we want to focus is the exceptional Lie group $G_2$, which is the automorphism group of the $\IR$-algebra homomorphisms of the octonions $\mathbb{O}$. From all the possible $G_2$ flag manifolds, we are particularly interested in the spaces $G_2/U(1) \times U(1) $ and $G_2/U(2)_{\pm}$. In the following diagram of fiber bundles we appreciate how they are related:
\begin{equation}
\label{Diagram: fiber_bundles}
	\xymatrix{
		&\ar[dl]_{\rho_{4}} G_2/U(1) \times U(1) \ar[dd]^{\rho_{3}}  \ar[dr]^{\rho_{5}}  &   	\\
		G_2/U(2)_+\ar[dr]_{\rho_{1}}&   								 						&  G_2/U(2)_-  \ar[dl]^{\rho_{2}}  	\\
			 									 & 		 G_2/SO(4) .	 									&																						
	}
\end{equation}
In fact, those flag manifolds are precisely the only three twistor spaces of the homogeneous space $ G_2/SO(4)$, see 
\cite[Sec. 2.3]{svensonwood}. 

\medskip
Previous to this work, several authors studied the integral cohomology of all the homogeneous spaces 
appearing in diagram \ref{Diagram: fiber_bundles}. In sections \ref{sec: G2 Flag Manifolds} and \ref{sec: Calculations F_HIndex} we calculate their cohomology with $\IF_2$ coefficients, and the Stiefel-Whitney classes of the fiber bundles $\rho_1$ and $\rho_2$.
Moreover, because of these calculations, in section \ref{sec: Calculations F_HIndex} we show that the cohomology rings with $\IF_2$ coefficients of $G_{2}/U(2)_{+}$ and $G_{2}/U(2)_{-}$ are isomorphic. For that reason we will not distinguish between them and we will just write $G_{2}/U(2)_{\pm}$.

Considering the action of $\IZ_2$ on $G_{2}/U(2)_{\pm}$ by complex conjugation, in section \ref{subsec: FadellHusseini rho1 and rho2} we prove our first main result.

\begin{customthm}
{\ref{thm: BorelCohom G2/U2 & FH rho1&2}} 
\label{MainResult 1}
	The Fadell-Husseini index of $\rho_1$ and $\rho_2$ is given by
	\[
		\Index{ \rho_1 }{G_2/SO(4)}{\IZ_2}{\IF_2} =	\Index{ \rho_2 }{G_2/SO(4)}{\IZ_2}{\IF_2} =	\free{  t^3 + u_2 t + u_3  },
	\]
	where $H^*(B\IZ_2; \IF_2 ) = \IF_2[t]$ with $\deg(t) = 1$. Consequently, the Borel cohomology of  $ G_{2}/U(2)_{\pm} $ 
	is given by 
	\[ 
		H^*( \CBorel{\IZ_2}{G_2/U(2)_{\pm}} ; \IF_2 ) = 	
					H^*( B\IZ_2 \times  G_2/SO(4) ; \IF_2 )  \Big/ \free{t^3 + u_2 t + u_3}. 
	\]
\end{customthm}

On the other hand, let us consider the tautological bundle over the oriented Grassmann manifold $\OGr{k}{n}$: 
\[
	\gamma_k^d = \big( E(\gamma_k^d ), \; \OGr{k}{d}, \;  E(\gamma_k^d ) \xrightarrow{\pi} \OGr{k}{d}, \; \IR^k \big),
\]
and the $n$-fold product associated to sphere bundle $s\gamma_k^d$: 
\[
	(s\gamma_k^d)^n = \big( E(s\gamma_k^d )^n , \; {\OGr{k}{d}}^n, \;  E(s\gamma_k^d )^n \xrightarrow{ (s\pi)^n } {\OGr{k}{d}}^n, 			\; (S^{k-1})^n \big).
\]
Given the pullback bundle of $(s\gamma_k^d)^n$ along the diagonal map $\Delta_n$:
\[
	\xymatrix{
	E\BigC{ 
	\Delta_n^* \bigC{(s\gamma_k^d)^n } } 
	\ar[r] \ar[d]&	E\bigC{s\gamma}^n \ar[d]^{(s\pi)^n} \\
	\OGr{k}{d}	\ar[r]^{\Delta_n} &		\OGr{k}{d}^n,}
\]   
the total space of $\Delta_n^* \bigC{(s\gamma_k^d)^n } $ has been an excellent candidate to be the configuration spaces of several geometric problems which uses the \emph{configuration space/test map scheme}. This means that, the more we know about the pullback bundle $\Delta_n^* \bigC{(s\gamma_k^d)^n }$, the better chance we have of solving any related problem. 

\medskip
Motivated by this situation, and by some calculations over Grassmann manifolds presented in \cite{Barali2018} and \cite{basu2020cohomology}, in this work we will study the following situation: 
Let us consider the sphere bundle of the $n$-fold product $\gamma^n$, where $\gamma$ represents the tautological bundle over $\OGr{4}{7}$
\[
	s\gamma = \big( E(s\gamma), \;  \OGr{4}{7},  \; 
			E(s\gamma) \xrightarrow{s\pi} \OGr{4}{7}, \;  S^3 \big).
\] 
Since there is an embedding $i$ between $G_2/SO(4)$ and $\OGr{3}{7}$, we consider the map $f = c \circ i$, where $c \colon \OGr{3}{7} \to \OGr{4}{7}$ assign to every linear subspace its orthogonal complement. 
We are interested on the pullback bundle of $(s\gamma)^n$ along the map $\Delta_n \circ f \circ \rho_j$, where $\Delta_n: \OGr{4}{7} \to \OGr{4}{7}^n$ is the diagonal map:
\[ 
	\zeta_n = \big( E(\zeta_n) = \calS_{\gamma}^{n}, \; G_2/U(2)_{\pm},  \; 
			\calS_{\gamma}^{n} \xrightarrow{\phi_n} G_2/U(2)_{\pm}, \; (S^3)^n \big).
\]
An example of this kind of constructions appears in \cite{basu2020cohomology}. The  total space $\calS_{\gamma}^n$ considers now collections of unit vectors inside $4$-dimensional subspaces of $\IR^7$, which are exactly the orthogonal complement of spaces which we usually called \emph{associative subspaces}. The specific details about the construction of  $\calS_{\gamma}^n$, as well as some topological properties of the bundle $\zeta_n$,  are discussed further in section \ref{subsec: FadellHusseini pullback}. We will prove then the following results:

\begin{customthm}{\ref{thm: FH Pullback}} 
\label{MainResult 3}
	Consider the action of $\IZ_2^{n}$ on $S_{\gamma^{\perp}}^{n}$ where the each summand in $\IZ_2^{n}$ acts antipodally on 
	a unit vector. Then the Fadell-Husseini index of $\Map{\phi_{n}}{ \calS_{\gamma}^{n} }{ G_2/U(2)_{\pm} } $ is given by
	\[
		\Index{ \phi_{n} }{G_2/U(2)_{\pm}}{\IZ_2^{n}}{\IF_2} 
			= \free{ y^2 + y t_1^2 + t_1^4, \ldots , y^2 + y t_{n}^2 + t_{n}^4 }.
	\]
	where $H^*(\IZ_2^{n}; \IF_2 ) = \IF_2[t_1, \ldots, t_{n}]$ with $\deg(t_1) = \cdots = \deg(t_{n}) = 1$.
\end{customthm}

At the end of the work we present an application of our computations to a mass partition problem, a classical problem in discrete geometry.  

\subsection*{\Large Acknowledgements}
The  first  author thanks  Manuel Sedano for enlightening conversations concerning homogeneous spaces and their almost  complex structures, and Gregor Weingart about the cohomology and characteristic classes of $G_{2}$-homogeneous manifolds. The first  author  thanks  the  support  by  PAPIIT  Grant IA 100119 and  IN 100221, as  well  as a  Sabbatical  Fellowship by DGAPA-UNAM for  a  stay  at the  University  of  the  Saarland and  the  SFB TRR 195 Symbolic Tools in Mathematics and their Application.  

\medskip
The second author was supported by the UNAM Posdoctoral Scholarship Program (DGAPA). The authors thank the support of CONACYT trough grant CF-2019 217392 and Hood Chatham for his support concerning the latex package 
\emph{ spectralsequences}.


\section{ \Large The exceptional Lie group $G_{2}$.}
\label{sec: The exceptional Lie group G2}

In this section we will recall three equivalent definitions of the  real  form of the exceptional Lie group $G_{2}$. Let us fix now the notation. General references for the upcoming discussion include \cite{baez} and \cite{draymanogue}. 

\subsection{Octonions algebra}
\label{subsec: octonions}

Given a normed division algebra $A$, the Cayley-Dickinson construction creates a new algebra $A^{'}$ with elements 
$(a,b) \in A^2$ and conjugation $(a,b)^{*}= (\bar{a}, -b)$ . The addition in $A^{'}$ is done component-wise, and multiplication goes like
\[ (a,b)(c,d) = ( ac-d \bar{b} , \bar{a}d + cb ), \]
where juxtaposition indicates multiplication in $A$. An equivalent way to define the new algebra $A^{'}$ is to add an independent square root of $-1$, $i$, that multiplies the second named element on each pair $(a,b)$. Now the conjugation in $A^{'}$ uses the original conjugation of $A$ and $i^{*} = -i$. Then the construction becomes an algebra of elements $a+ib$ for some $a,b \in A$.

\medskip
Starting with $\IR$, the complex numbers are defined via the Cayley-Dickinson construction to be elements 
$a+ib$, with $a$ and $b$ real numbers. Similarly, the quaternions are generated as a real algebra by $\{1, i, j, k \}$, subject to the relations $i^{2}=j^{2}= k^{2}= ijk=-1$. Having this in mind, we get the following definition

\begin{definition}
\label{def:octonions}
	The octonions are generated as a quaternionic algebra via the Cayley-Dickinson construction as 
	$\IH \oplus  \IH[l]$,  where $l$ denotes an independent square root of $-1$. 
\end{definition}

Even if the Cayley-Dickinson construction is an excellent method to produce other normed division algebras, we lost very nice properties in the process. The multiplication in $\IO$ turns out to be non-commutative and non-associative. As a real vector space, $\IO$ is generated by $\{ 1, i,j,k,l, li, lj, lk \} = \{1, e_{1}, \ldots, e_{7} \}$,  where $e_{1}, e_{2}, \ldots e_{7}$ are  imaginary units, which square to $-1$,  switch sign under complex conjugation and anticommute. They span the seven dimensional real subspace which we will denote as the purely imaginary part of the octonions, $\ImO$. 
 
\medskip
Denote by $\cdot$ the product furnishing the octonions with the structure of real division algebra. 
The product $\cdot$ determines a cross product in $\IO$, expressed as
\[ 
	x \times y:= \frac{1}{2} (x\cdot y - y\cdot x),
\] 
which contains all the non-trivial algebraic information about the octonions, and turns $\ImO$ into an algebra. In case that $x$ and $y$ are imaginary cotonions, the relation $x \cdot y  + (x,y)= x \times y$ holds, where $(x,y)$ represents the scalar product. As will be discussed below,  the algebra $(\ImO, \times)$ provides a new definition of the group $G_{2}$.

\subsection{The exceptional Lie group $G_{2}$: three equivalent definitions}
\label{subsec: G2}

Now that we are familiar with the $\IR$-algebra $\IO$, we are ready to introduce the main ingredient of this work.

\begin{definition} 
\label{def:g2}
	The exceptional Lie group $G_{2}$ is defined to be the group of all $\IR$-algebra automorphisms of 
	the octonions. 
\end{definition}

We will consider an alternative but equivalent definition of $G_2$: Given two orthogonal and unit imaginary octonions $x$ and $y$, the cross product $x \times y$  is orthogonal to both of them, and the subalgebra generated by $\{1, x, y, x \times y  \}$ is isomorphic to $\IH$. 
If we consider another unit imaginary octonion $z$, orthogonal to the subspace generated by $\{1, x, y, x \times  y\} \cong \IH$,  then the subalgebra over the quaternions $\IH \oplus \IH[z]$ is isomorphic to the octonions. This has the consequence that an element $g \in G_{2}$ can be characterized by prescribing three unit imaginary  octonions $\{x, y, z\}$, with $z$ orthogonal to 
$x$, $y$ and $x\times y$. The element $g$ is then the unique automorphism of the imaginary part of the octonions, which sends $\{ x,y,z\}$ to $\{i,j,l\}$. Then the exceptional Lie group $G_{2}$ is defined to be the automorphism group of the algebra 
$({\rm Im }\, \IO, \times) $.

\medskip
On $\ImO$ one can define a cross product three-form on the generators $e_{i}, e_{j}, e_{k}$ as 
\[
	\phi(e_{i}, e_{j}, 	e_{k}) = f^{ijk} ,
\] 
where $e_{i} \times e_{j} = \underset{k}{\sum }f^{ijk} e_{k}$. If we consider the dual basis $\{w^{i} :=e_{i}^{*}\}$ of the generators 
$\{e_{i}\}$ given above, then the three-form $\phi$ is given by  
\[ 
	\phi= w^{123} - w^{145} - w^{167} - w^{246} + w^{257} - w^{347} - w^{356},
\]
where the notation $w^{ijk}$ denotes the wedge product $e^{i} \wedge e^{j} \wedge e^{j}$. Notice that $\phi$ encodes the multiplicative structure of the cross product in the imaginary part of the octonions. Then, using the above considerations, we obtain a third equivalent definition of the exceptional Lie group $G_{2}$ as follows:
\[ 
	G_2 = \{ g \in GL(\ImO) \mid g^{*} \phi = \phi \}. 
\]

\medskip
The following result summarizes the previous discussion about the real form of $G_2$, besides some of its properties presented by Bryant in \cite[Theorem 1]{bryant}.

\begin{theorem}
\label{theorem:g2isomorphic}
	The real form of the group $G_2$ has dimension $14$. It is simple, simply connected, and is isomorphic to one and hence to 	
	all of the following Lie groups: 
	\begin{enumerate}
		\item The group of all $\IR$-algebra automorphisms of the octonions.
		\item The automorphism group of the subalgebra $(\ImO, \times)$. 
		\item The subgroup of $SO(7)$ which preserves the cross product three form. 
	\end{enumerate}
\end{theorem}


\section{ \Large $G_{2}$ flag  manifolds}
\label{sec: G2 Flag Manifolds}

The flag manifolds on $G_2$ have received recently a lot of attention from several viewpoints. From riemannian geometry 
(\cite{bryant}, \cite{svensonwood}), algebraic topology relevant to global analysis \cite{akbulutkalafat}, complex and K\"ahler geometry (\cite{kotschickthung}, \cite{semmelmannweingart}), and from the  classical  computations  of  their  characteristic  classes  via  representation theory by Borel and Hirzebruch (\cite{borelhirzebruch}, \cite{gramanegreiros}). The combination of the above facts, excellently produced in \cite{kotschickthung}, aroused our interest in the subject and gave us the idea to calculate the Fadell-Husseini index of some fiber bundles over $G_2/SO(4)$ with total space a $G_2$-flag manifold. 

\medskip
In this section we introduce some flag manifolds associated to the exceptional Lie group $G_2$. Since we are particularly interested on the $G_2$-flag manifolds which fiber over $G_2/SO(4)$, we are going to stat introducing the 
$G_2$-homogeneous space $G_2/SO(4)$ in order to define the fiber bundles for which we want to calculate their Fadell-Husseini indices.

\subsection{The space of associative subspaces  $G_{2}/SO(4)$. } 
\label{subsec: G2/SO(4)}

Now that we have introduced the exceptional Lie group $G_2$ with some equivalent definitions, we are ready to study the $G_2$-homogeneous space that we are going to use. 

\begin{definition}
\label{def:associative}
	A $3$-dimensional subspace $\xi \subset \ImO$ is said to be \emph{associative} if it is the imaginary part of a subalgebra 
	isomorphic to the 	quaternions. 
\end{definition}

Notice that the subspace $\xi $ acquires a canonical orientation from that of the quaternions. On the other hand, in terms of the three-form defined in \ref{subsec: G2}, we can also define an associative subspace as the $3$-dimensional real subspace of $\IR^{7} \cong_{\IR} \ImO$ in which the three form $\phi$ agrees with the volume form ${\rm Vol}(\xi)$. 

\medskip
The group $G_2$ stabilizes $1$, and since the scalar product is the real part of the octonionic multiplication $\cdot$, it acts by isometries on $\ImO$. Since also preserves the vector product, it preserves orientation so that
\[ 
	G_2 \subset SO(\ImO) \cong SO(7). 
\]
Every associative subspace $\xi$ admits an orthonormal basis $\{ e_1,e_2,e_3 \}$ with $e_1 \times e_2 = e_3$. Then, given such a triple $\{ e_1,e_2,e_3 \}$, there exists a unique automorphism in $\ImO$ taking the ordered triple $\{i,j,l\}$ to it. 
This induces  a transitive action of $G_2$ on the Grassmannian of associative $3$-dimensional subspaces of $\ImO \cong \IR^7$, with stabilizer $SO(4)$. Then $G_2/SO(4)$ is the set of all associative $3$-dimensional subspaces of $\ImO \cong \IR^7$. Finally, since every associative subspace has a canonical orientation, there is an embedding of $G_2/SO(4)$ in the $3$-dimensional oriented Grassmannian manifold $\OGr{3}{7}$, with $\IR^7 \cong \ImO$.

\medskip
Borel and Hirzebruch studied in \cite{borelhirzebruch} the cohomology with $\IF_2$-coefficients and the Stiefel-Whitney classes of $G_2/SO(4)$. They proved the following result:

\begin{proposition}
\label{prop: Borel_Hirzebruch G2/SO(4)}
	The homogeneous space of associative $3$-dimensional real subspaces, denoted by $G_{2}/SO_4$,  is an $8$-dimensional  				manifold for which $H^*( G_2/SO(4) ;\IF_2)$ is generated by two elements $u_2, u_3$ of degrees $2$ and $3$ respectively, with 
	the relations 
	\[ 
		u_2^3 = u_3^2 \quad \text{ and } \quad u_3 u_2^2 = 0.
	\]  
	Also, the Stiefel-Whitney classes of $G_2/SO(4)$ are non-zero only in dimensions $0,4,6$ and $8$.
\end{proposition}

On the other hand, Shi and Zhou in \cite[Section 10]{shi2014characteristic}, Thung in \cite[Section 2]{kotschickthung}, and Akbulut and Kalafat \cite[Theorem 10.3]{akbulutkalafat}, studied the integral cohomology of $G_2/SO(4)$, and the relations between the corresponding generators to write it as a truncated polynomial algebra.

\begin{proposition}
\label{prop:cohomSo4}
	 The integral cohomology of $G_2/SO(4)$ is a truncated polynomial algebra generated by two elements $a$ and $b$, of degree 
	 $3$ and $4$ respectively, subject to the relations 
	\[ 
		\{ 2b=0,\; b^{3}=0, \; a^{3}=0, \; ab=0  \}. 
	\] 
\end{proposition}

\subsection{The six  dimensional sphere $G_{2}/SU(3)$. } 
\label{subsec: G2/SU3}

We are going to start with a $G_2$-flag manifold that, even if it do not fibre over $G_2/SO(4)$, it is a good first example of this type of spaces. 

\medskip
Remember that, by theorem \ref{theorem:g2isomorphic}[part ii], $G_2$ can be characterized by triples $\{x, y, z\}$ of mutually orthogonal unitary vectors in $\ImO$, where $z$ is also orthogonal to $x \times y$. This means that the six dimensional sphere, which can be seen as the unitary vectors in the seven dimensional real subspace $\ImO$, carries a transitive action of $G_2$.

\medskip
To determine the isotropy group of an element $l \in S^6$, consider the subspace $V$ defined as the orthogonal complement of $l$ inside $\ImO$, and the complex structure on $V$ given by the left multiplication by $l$. The identification with $\IC^3$, equipped with its stardand Hermitian scalar product, induces a scalar product on $V$ defined as 
\[
	\langle v, w \rangle_{V}= (v,w)+ l(v,lw),
\]
where $(-,-)$ denote the standard real product. Since $g \in (G_2)_{l}$ preserves $(-,-)$ , it also preserves $\langle -, - \rangle_{V}$, and $(G_2)_{l} \subset U(V) \cong U(3)$. Calculating the determinant of $g \in (G_2)_{l}$, can be proved that actually the isotropy group is  isomorphic to $SU(3)$. We have obtained the following result:

\begin{proposition}\label{prop:S6}
	There is a transitive action of $G_2$ on $S^6$ with isotropy group isomorphic to $SU(3)$, i.e. $S^6 \cong  G_2/SU(3)$.
\end{proposition}

\subsection{The  space  of complex coassociative $2$-planes $G_{2}/U(2)_{+}$.}
\label{subsec: G2/U2+ }

We will consider now the complexification of the purely imaginary subspace of the octonions, which is isomorphic as complex vector space $\IC^{7}$, in symbols $\ImO \otimes_{\IR} \IC \cong_{\IC} \IC^{7}$.

\begin{definition}
	We call a complex $2$-dimensional subspace $W \subset \IC^7$ \emph{coassociative} if $v \times w = 0$ for all $v,w \in W$
\end{definition}

Notice that the coassociative subspace $W$ is automatically isotropic. Also, the associative three-form $\phi \otimes {\rm id}$ defines a complex three-form $\phi_{\IC}$ on the complex space $\IC^{7}$. So, in terms of the three-form defined in \ref{subsec: G2}, a complex vector subspace $W \subset \mathbb{C}^{7}$ is called coassociative if the complexified form $\phi_{\IC}$ vanishes on $V$. We will consider coassociative complex lines ($1$-dimensional complex subspaces) and coassociative planes ($2$-dimensional  complex subspaces).

\medskip
On the other hand, let $J$ be an orthogonal almost complex structure on $\xi^{\perp}$ with $(1,0)$-space $W$, 
where $\xi$ is an associative subspace in $\ImO$, and $W$ is the eigenspace associated to the eigenvalue $i$ of $J$. 
Choose now an orthonormal basis $\{ e_1, e_2, e_3, e_4 \}$ of $\xi^{\perp}$ with $J(e_1) = e_2$ and $J(e_3)=e_4$.  
We say that $J$, or the corresponding $(1,0)$-space $W$, is positive (resp. negative) according as the basis of $\xi^{\perp}$ is positive (resp. negative). 

\medskip
The following result, which connects the notion of complex-coassociative and positive subspaces, is proved in 
\cite[Lemma  2.2, page 293]{svensonwood}. 

\begin{lemma}\label{lemma: coassociative}
Let $\xi$ be an associative $3$-dimensional subspace in $\ImO$, and let $W \subset \xi^{\perp} \otimes \IC$ be a maximally isotropic subspace. Then $W$ is positive if and only if it is complex coassociative.  
\end{lemma} 

Consider the space of all complex coassociative $2$-dimensional subspaces of $\IC^{7}$. Since such space has a transitive action of $G_2$ with stabilizer isomorphic to $U(2)$, then we have the following definition.

\begin{definition}
	The space of complex coassociative $2$-planes in $\ImO \otimes \IC$ will be denoted by $G_{2}/U(2)_{+}$.
\end{definition}

The reason for the notation comes from lemma \ref{lemma: coassociative}. Moreover, there is a map 
\[
	\rho_{1}:G_{2}/U(2)_{+} \to G_{2}/ SO_{4},
\] 
given by $W \mapsto (W \oplus \bar{W})^{\perp}$, which exhibits $G_{2}/U(2)_{+}$ as the total space of a locally trivial smooth fibration with fiber $\IC P^{1} \approx S^{2}$. By lemma \ref{lemma: coassociative} the fiber of any $\xi \in G_2/SO(4)$ is all positive maximally isotropic subspaces of $\xi^{\perp} \otimes \IC$, equivalently, all positive orthogonal almost complex structures on $\xi^{\perp}$. 
Actually, by \cite[section 4.4]{svensonwood} there is an identification of $G_2/U(2)_{+}$ with a quaternionic twistor space of $G_{2}/SO_{4}$.

\medskip
As is exposed in \cite[Prop. 3]{kotschickthung}, the cohomology of $G_{2}/U(2)_{+}$ with integral coefficients is generated by classes $g_{i} \in H^{2i}(G_{2}/U(2)_{+}, \IZ)$, for $i=1, \ldots, 5$, and the multiplicative structure is determined by the relations 
\[ 
	g_{1}^{2}=3g_{2}, \, g_{1}g_{2}= 2g_{3}, \, g_{2}^{2} = 2g_{4}, \, g_{1}g_{4}=g_{5}. 
\]
The cohomology with coefficients in $\IF_2$ will be calculated in subsection \ref{subsec: FadellHusseini rho1 and rho2}.

\subsection{The  space  of complex non-coassociative $\mathbf{2}$-planes $\mathbf{G_{2}/ U(2)_{-}}$.}
\label{subsec: G2/U2- }

\begin{definition}\label{def:g2-}
	The space of all $2$-planes in $\ImO \otimes \IC$ which are not complex coassociative will be denoted by $G_{2}/U(2)_{-}$.
 \end{definition}

Similarly to $G_{2}/U(2)_{+}$, the reason for the notation comes from lemma \ref{lemma: coassociative}. A $2$-dimensional real subspace $W$, which is a $(1,0)$-space for a negative orthogonal almost complex structure in $\xi^{\perp}$,  is said to be negative and is still a maximally isotropic subspace of $\xi^{\perp} \otimes \IC$.

\medskip
Alternatively, we can think the space $G_2/U(2)_{-}$ as follows: Let $Q^5$ be the complex quadric 
$\bigg\{ [z_{1}, \ldots , z_{7}] \in \IC P^{6} \mid \underset{i=1}{\overset{7}{\sum}} z_{i}^{2}=0   \bigg\}$ consisting of all $1$-dimensional isotropic subspaces of $\ImO \otimes \IC$. 
There is a $G_2$-equivariant isomorphism from $Q^5$ to $G_2/U(2)_{-}$ given by $\ell \mapsto \ell^{\perp} \cap \ell^{a}$, with inverse $W \mapsto W \times W $, where $\ell^{a} = \{ x \in \ImO \otimes \IC \mid x \times \ell = 0 \}$ is the annihilator of $\ell$.

\medskip
Since $\OGr{2}{7}$ is also diffeomorphic to the complex quadric $Q^5$, we can identify every element in $G_2/U(2)_{-}$ as an oriented plane $P$ with oriented, orthonormal basis $\{x,y\}$. 
Then there is a well defined map $\rho_{6}:G_2/U(2)_{-} \to S^6$ which sends $P \mapsto x \times y := xy$. 
Actually, the oriented plane $P$ can be identify, via the standard almost complex structure of $S^6$, with a complex line in $T_{xy}S^6$. Recall that the standard almost complex structure  on $S^6$ at the point $v \in S^6$ is given by the left-multiplication $L_v$.
On the other hand, given an element $k \in S^6$, there is a oriented orthonormal basis $\{x,y\}$ of a complex line in $T_kS^6$ that satisfies $k = xy$. This means that the fiber over an element $v \in S^6$ is precisely $\mathbb{P}(T_vS^6)$. This exhibits $G_2/U(2)_{-}$ as $\mathbb{P}(TS^6)$ with $\rho_{6} $ as the base point projection. We can also think of $\rho_{6}$ as a fibration with fibers diffeomorphic to $\IC P^{2}$.  
 
\medskip
Finally, the isomorphism $TS^6 \cong T^{*}S^6$ as real vector bundles induces a diffeomorphism 
$\mathbb{P}(TS^6) \cong \mathbb{P}(T^{*}S^6)$.  The following result resumes all the diffeomorphic definitions of $G_{2}/ U(2)_{-}$. See \cite[Prop. 8]{kotschickthung}. 

\begin{proposition} \label{prop:u2-}
	The following $10$-dimensional manifolds are all diffeomorphic to each other:
	\begin{enumerate}
		\item The space of all $2$-planes in $\ImO \otimes \IC$ which are not complex coassociative.
	
		\item The Grassmannian $\OGr{2}{7}$ of oriented $2$-planes in $\IR^7$.	
		
		\item The complex quadric 
		\[
			Q^5=\bigg\{ [z_{1}, \ldots , z_{7}] \in \IC P^{6} \mid \underset{i=1}{\overset{7}{\sum}} z_{i}^{2}=0   \bigg\}
		\] 
		
		\item The projectivization of the tangent and cotangent bundle $\mathbb{P}(TS^6)$ and $\mathbb{P}(T^{*}S^6)$ for any 
		almost complex structure in $S^6$. 
	\end{enumerate}
\end{proposition}

The cohomology with integral coefficients of $G_{2}/U(2)_{-}$ is calculated in \cite[Prop. 11]{kotschickthung}. It is the quotient of 
a polynomial algebra with generators $x$ in degree $6$, and $y$ in degree  $2$, satisfying the relations $x^{2}=0 \text{ and } y^{3}= -2x$. Unlike $G_{2}/U(2)_{+}$, using the bundle $\rho_{6}$ we can calculate the cohomology ring of $G_{2}/U(2)_{-}$ with $\IF_2$ coefficients as follows:

\begin{proposition}\label{prop: Cohom G2/U2-}
	The cohomology of $G_{2}/U(2)_{-}$ with coefficients in $\IF_2$ is given by 
	\[
		H^{*}(G_{2}/U(2)_{-}, \IF_2) = \IF_2[ x,y ] \big/  \free{ x^{2}, y^{3} },
	\]
	where $deg(y) = 2$, $deg(x)=6$.
\end{proposition}
\begin{proof}
	Consider the fiber bundle 
	\[ 
		\CP{2} \hookrightarrow  G_2/U(2)_{-} \rightarrow S^6 
	\]
	and the corresponding Serre spectral sequence. Since $S^6$ is simply connected, there are no local coefficients and the 
	$E_2$-term is given by 
	\[  
		E_2^{p,q} = H^p \big( S^6 ;H^q( \mathbb{CP}^2 ;\IF_2 ) \big).
	\] 
	Let us denote the cohomology of the fiber $\CP{2}$ and the base space $S^6$ as follows: 
	\[
		H^*(S^6;\IF_2) = \IZ_2[x]/\langle x^2 \rangle \; \text{ and } \; H^*(\mathbb{CP}^2;\IF_2) = \IF_2[y]/	\langle y^3 \rangle. 
	\] 	
	For an illustration of the $E_2$-term see Figure \ref{Figure: sseq 1}. Then, since all 
	the possible differentials are trivial, $E_2^{p,q} \cong E_{\infty}^{p,q}$ and 
	\[
		 H^{n}(G_{2}/U(2)_{-}, \IF_2) = \IF_2[ x,y ] \big/  \langle x^{2}, y^{3}  \rangle. 
	\]
	\begin{figure}[h]
		\centering
		\includegraphics[scale=1.1]{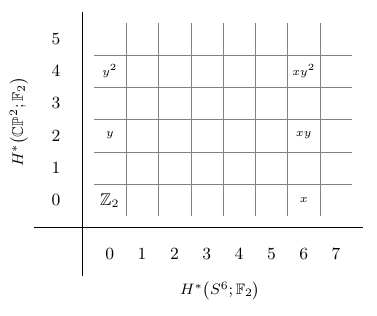}
		\caption{ $E_2^{p,q} = H^p \big( S^6 ;H^q( \mathbb{CP}^2 ;\IF_2 ) \big) \Rightarrow H^*( G_2/U(2)_- ;\IF_2 ) $.}
		\label{Figure: sseq 1}
	\end{figure}
\end{proof}

Similar to $G_{2}/U(2)_{+}$, there exists a map 
\[
	\rho_{2} \colon G_{2}/U(2)_{-} \to G_{2}/ SO_{4},
\] 
given by $W \mapsto (W \oplus \bar{W})^{\perp}$, which exhibits $G_{2}/U(2)_{-}$ as the total space of a locally trivial smooth fibration with fiber $\IC P^{1} \approx S^{2}$. By lemma \ref{lemma: coassociative}, the fiber at any $\xi \in G_2/SO(4)$ consist of all negative maximally isotropic subspaces of $\xi^{\perp} \otimes \IC$, or equivalently, all negative orthogonal almost complex structures on $\xi^{\perp}$. 
%

\subsection{ The full flag  manifold $\mathbf{G_{2}/ U(1)\times U(1)}$. }
\label{subsec: G2/U(1)xU(1) }

Given a complex isotropic line $\ell \subset \ImO \otimes_{\IR } \IC$, consider the annihilator $ \ell ^{a}$, that is the subspace of 
$\ImO \otimes_{\IR} \IC$ described as $\{ x \in \ImO \otimes_{\IR} \IC \mid x \times \ell=0\}.$ Notice that this is a complex $3$-dimensional isotropic subspace of $\ImO \otimes_{\IR} \IC$. Since the maximal torus in $G_{2}$ is of rank $2$, the complete flag manifold $ G_{2}/U(1)\times U(1) $ is a smooth complex manifold of dimension $6$ that can be described as follows: The space of pairs $(\ell, D)$ where $\ell$ is a complex isotropic line, $D$ is a $2$-plane containing $\ell$, and both are contained in 
$\ell^{a}$.

\medskip
For every pair $(\ell, D)$ in $G_2/U(1)\times U(1)$, we write $D = \ell \oplus q$, where $q$ is the orthogonal complement of 
$\ell$. Then we get a fibration 
\[
	\rho_{3} \colon G_2/U(1)\times U(1) \to G_2/SO(4)
\] 
given by $ (\ell, D) \mapsto \xi $, where $\xi \otimes_{\IR} \IC = q \oplus \bar{q} \oplus (q \times \bar{q})$. The map $\rho_{3}$ factors through a fibration over $G_2/U(2)_{\pm}$, which sends $(\ell,D)$ to the positive (resp. negative) maximally isotropic subspace of $\xi^{\perp} \otimes_{\IR} \IC$ which contains $\ell$, and $\rho_{1}$ (resp. $\rho_{2}$) defined above.

\begin{equation}
\label{Diagram: fiber_bundles2}	
		\xymatrix{
		&\ar[dl]_{\rho_{4}} G_2/U(1) \times U(1) \ar[dd]^{\rho_{3}}  \ar[dr]^{\rho_{5}}  &   										&\\
		G_2/U(2)_+\ar[dr]_{\rho_{1}}&   				 &  G_2/U(2)_- \cong Q_5 \ar[dl]^{\rho_{2}} \ar[dr]^{\rho_{6}} 	&\\
			 									 & 		 G_2/SO(4) 	 									&												               	& S^6}
\end{equation}

To summarize, all the maps $\rho_{i}$ in the commutative diagram \ref{Diagram: fiber_bundles2} are fibrations, with fiber diffeomorphic to $\IC P^{1} \cong S^{2}$ (except $\rho_{3}$ and $\rho_{6}$). 
Furthermore, $\rho_1$, $\rho_2$ and $\rho_3$ are fiber bundles over $G_2/SO(4)$.  
In the case of $\rho_{6}$ the fiber is diffeomorphic to $\IC P^{2}$. This means that, using the map $\rho_{5}$, we can also calculate the cohomology of $G_{2}/U(1) \times U(1)$ with $\IF_2$ coefficients.

\begin{proposition}\label{prop: Cohom G2T2}
	The cohomology of $G_{2}/U(1) \times U(1)$ with $\IF_2$ coefficients is given by 
	\[ H^{n}(G_{2}/U(1) \times U(1), \IF_2) =
					  \begin{cases}
					    \IZ_2       & \quad \text{if } n =0,12 \\
					    \IZ_2 \oplus \IZ_2  & \quad \text{if } n =2,4,6,8,10 \\
					    0            & \quad \text{otherwise.}  
					  \end{cases}
\]
\end{proposition}
\begin{proof}
	Consider the fibration
	\[ 
		S^2\hookrightarrow G_{2}/U(1) \times U(1) \xrightarrow{\rho_5}  G_2/U(2)_{-}, 
	\]
	and the corresponding Serre spectral sequence, where the cohomology of the fiber is described as 
	$H^*(S^2;\IF_2) = \IF_2[z]/z^2$. Since $\pi_1( G_2/U(2)_{-} )$ is trivial, there are no local coefficients and the $E_2$-term is 
	given by 
	\[
		E_2^{p,q} = H^p \big( G_2/U(2)_{-} ;H^q( S^2 ;\IF_2 ) \big). 	
	\]
	For an illustration of the $E_2$-term see Figure \ref{Figure: sseq 2}. The rest of the proof is similar to the one in 
	\ref{prop: Cohom G2/U2-}, since all the possible differentials are trivial. 
	\begin{figure}[h]
		\centering
		\includegraphics[scale=1.1]{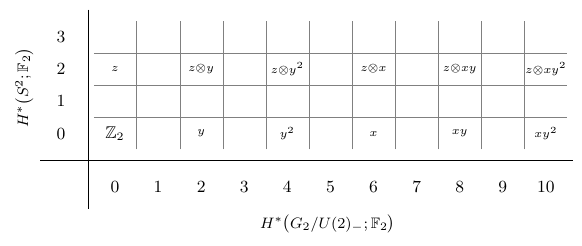}
		\caption{ $E_2^{p,q} = H^p \big( G_2/U(2)_{-} ;H^q( S^2 ;\IF_2 ) \big) \Rightarrow H^*( G_{2}/U(1) \times U(1);\IF_2 ) $.}
		\label{Figure: sseq 2}
	\end{figure}
\end{proof}
 

\section{ \Large Calculations of Fadell-Husseini Index.}
\label{sec: Calculations F_HIndex}

We start by recalling the definition of the Fadell-Husseini index. 

\subsection{The Fadell-Husseini Index}
\label{subsec:FadelHusseini}

Let $G$ be a finite group, and let $R$ be a commutative ring with unit. For a $G-$ equivariant map $p \colon X \to B$,
the \emph{Fadell-Husseini index} of $p$ with coefficients in $R$ is defined to be the kernel ideal of the following induced map 
\begin{align*}
	\ind_G^B(p;R) & = \ker \big( p* \colon H^*(EG \times_G B;R) \longrightarrow H^* (EG \times_G X;R)  \big)  \\
		& = \ker \big( p* \colon H_G^*(B;R) \longrightarrow H_G^* (X;R)  \big).
\end{align*}
Here $H^*_{G}( \cdot )$ stands for the Borel cohomology defined as the \v{C}ech cohomology of the Borel construction 
$EG \times_G \cdot$. 

\medskip
The most important property of the Fadell-Husseini index is the \emph{monotonicity} which states the following: If $p \colon X \to B$ and $q \colon Y \to B$ are $G-$equivariant maps, and $f \colon X \to Y$ is a $G-$equivariant map such that $p = q \circ f$, then 
\[ 
	\ind_G^B(p;R) \supseteq \ind_G^B(q;R).  
\]
The monotonicity of the index is usually the key ingredient to verify the existence of equivariant maps.
Is common to relate the Fadell-Husseini index with spectral sequences since, under certain conditions, 
by \cite[Theorem 5.19]{mccleary2001user} the Serre spectral sequence compute the kernel of induced homomorphisms in cohomology. 

\medskip
In the case when $B$ is a point we simplifly notation and write $\ind_G^B(p;R) = \ind_G(X;R)$. From here on we will be considering $\IF_2$ coefficients. For more details see \cite{fadellhusseini}

\subsection{Dold's argument }
\label{subsec: Dold argument}


In order to compute the Fadell-Husseini index of some of the fibrations exposed in section \ref{sec: G2 Flag Manifolds}, 
we present a result of Albrecht Dold \cite{dold1988parametrized} which allow us to deduce some differentials of the Serre spectral sequence associated to a particular kind of fiber bundles. Let us start with the general argument.

\medskip
Let $ E \xrightarrow{\pi} B \xleftarrow{\pi'} E' $ be vector bundles of fiber-dimension $n,m$ over the same paracompact space $B$, and let $\Map{f}{S(E)}{E'}$ be an odd map ($f(-x) = -f(x)$), where $S(E)$ is the total space of the sphere bundle associated to $\pi$, such that
\[ 
	\xyTriangle{S(E)}{f}{E'}{s\pi}{\pi'}{B}
\]
commutes. Let us define $Z_f=\{ x \in S(E) \mid f(x)=0\}$, where $0$ stands for the zero section of $\pi'$, and the projection maps 
\[ 
	S(E) 	\rightarrow \bar{S}(E) = S(E)/\IZ_2 \quad \text{and} \quad Z \rightarrow \bar{Z} = Z/\IZ_2 , 
\]
where we are considering the antipodal action.

\medskip
Cohomology $H^*$ is understood in the \v{C}ech sense with mod $2$ coefficients, and $H^*(B)[t]$ is the polynomial ring over $H^*(B)$ in one indeterminate $t$ of degree $1$. Since the antipodal action is fixed point free in $S(E)$ and $Z$, the projection maps $S(E) 	\rightarrow \bar{S}(E)$ and $Z \rightarrow \bar{Z}$ are $2$-sheeted covering maps. Their characteristic classes are denoted by $u$, resp. $u\mid \bar{Z}$, which can be replaced by the indeterminate $t$ and obtain an homomorphism of $H^*(B)$-algebras
\[ 
	\sigma \colon H^*(B)[t] \longrightarrow H^*( \bar{S}(E) ) \longrightarrow H^*(\bar{Z})  
\]
given by $t \longmapsto u \longmapsto u \mid \bar{Z}$. Dold proved the following result.

\begin{proposition}
\label{prop: Dold}
	If $q(t) \in H^*(B)[t]$ is such that $\sigma(q(t)) =  0$, then 
	\[
		q(t)W(\pi';t) = W(\pi;t)q'(t)
	\] 
	for some $q'(t) \in H^*(B)[t]$, where $W(\pi;t) = \sum_{j=0}^{n} w_j(\pi) \otimes t^{n-j}$. 
\end{proposition}	

The last theorem means that, under the last conditions, $W(\pi;t)$ divides $q(t)W(\pi';t)$. We show the effectiveness of this theorem in the following remark.

\begin{remark}
	Under the same hypothesis, consider the fiber bundles \\
	\begin{tabular}{ c c } 
	\begin{minipage}{0.4\textwidth}
	     		\[ \FiberBundle{S^{n-1}}{S(E)}{s\pi}{B} \]
	   		\end{minipage}
	     	  &
	     	 \begin{minipage}{0.4\textwidth}
	     	  	\[ \FiberBundle{\{0\}}{B \times \{0\}}{ {\rm proj}_1 }{B } \]	
	     	 \end{minipage}
	\end{tabular} 
	\\
	where $\IZ_2$ acts antipodally on $S(E)$ and trivial on $B$. Let  $\Map{f}{ S(E) }{ B \times \{0\}}$ be a 
	$\IZ_2$-equivariant map given by $f(e)=( \pi(e),0)$, such that the following diagram commutes:
	\[ 
		\xyTriangle{ SE }{f}{B \times \{0\}}	{s\pi}{ {\rm proj}_1}	{B}	
	\]
	Since $Z_f = S(E)$ and $W(proj_1,t) = 1$, if we consider $q(t)$ as the image of the transgression map 
	$d_n^{0,n-1}$ of the Serre spectral sequence associated to the sphere bundle 
	\[ 
		S^{n-1} \hookrightarrow  \CBorel{\IZ_2}{E}	\to B\IZ_2 \times B,	
	\]
	by proposition \ref{prop: Dold}, 
	\[ 
		q(t) = d_n^{0,n-1}(z) = \sum_{j=0}^{n} w_j(\pi) \otimes t^{n-j},
	\]
	where $H^*(S^{n-1}) = \IF_2[z]/\free{z^2}$.
\end{remark}

We are going to use Dold's argument in some of the computations that we present.

\subsection{The Fadell-Husseini index of $\rho_1$ and $\rho_2$.}
\label{subsec: FadellHusseini rho1 and rho2}

Following the idea presented in section \ref{subsec: Dold argument}, we are going to start studying the Stiefel-Whitney classes of $\rho_1$ and $\rho_2$ to deduce their corresponding Serre spectral sequences. 

\medskip
Since $G_2/SO(4)$ is a quaternionic Kähler manifold, and the bundles $\rho_1$ and $\rho_2$ are twistors fibrations over 
$G_2/SO(4)$, we can think of each $\rho_i$ as the associated sphere bundle of a $3$ rank vector bundle $\eta \colon H \to G_2/SO(4)$ with non-trivial Euler class.  
In this way, we are considering the Stiefel-Whitney classes of the vector bundle $\eta$.
For more detail about the bundle $S\eta$ see \cite{svensonwood} and \cite{salamon1982quaternionic}.

\medskip
Let us start describing the induced homomorphism $\rho_2^{*}$.

\begin{lemma}
\label{lemma: Induced map rho_2}
	The induced homomorphism in cohomology with $\IF_2$ coefficients, denoted by $\rho_2^*$, maps the class $u_2$ to $y$. 
\end{lemma}
\begin{proof}
	Let us start analyzing the Gysin exact sequence associated to the sphere bundle $\rho_2$,
	\begin{multline*}
	 \cdots \rightarrow H^{i-3}( G_2/SO(4) ;\IF_2 )  \xrightarrow{\smile w_3} H^{i}( G_2/SO(4) ;\IF_2 ) \rightarrow \\
	 	 \xrightarrow{\rho_2^*} H^{i}( G_2/U(2)_-  ;\IF_2 ) \rightarrow H^{i-2}( G_2/SO(4) ;\IF_2 ) \rightarrow \cdots .
	\end{multline*}
	Since $ H^{1}(G_2/SO_4, \IF_2) = 0 $, by the exactness of the sequence, 
	\[
		\rho_2^* \colon H^{4}( G_2/SO(4) ;\IF_2 ) \rightarrow H^{4}( G_2/U(2)_-  ;\IF_2 )
	\] 
	is a monomorphism. This means that $\rho_2^*(u_2^2) = y^2$ and consequently  $\rho_2^*(u_2) = y$.
\end{proof} 

\medskip
Now, since we already know the cohomology of $G_2/U(2)_{-}$, we are going to start with the characteristic classes of $\rho_2$.

\begin{lemma} 
\label{lemma: SW rho2}
	The Stiefel-Whitney classes of the fiber bundle $\Map{\rho_2}{G_2/U(2)_-}{G_2/SO(4)} $ are non-zero in dimensions 
	$0,2$ and $3$.
\end{lemma}
\begin{proof}
	Consider again the Gysin exact sequence applied to the sphere bundle $\rho_2$ 
	\begin{multline*}
		\ldots \rightarrow H^{i-3}( G_2/SO(4) ;\IF_2 )  \xrightarrow{\smile w_3} H^{i}( G_2/SO(4) ;\IF_2 ) \xrightarrow{\rho_2^*} 
		H^{i}(G_2/U(2)_- ;\IF_2 ) \rightarrow  \\ 	\rightarrow H^{i-2}( G_2/SO(4) ;\IF_2 ) \rightarrow \ldots
	\end{multline*}
	Since $H^3( G_2/U(2)_-; \IF_2 ) = 0$, we get an epimorphism 
	\[  
		H^0(G_2/SO(4) ;\IF_2) \twoheadrightarrow H^3(G_2/SO(4) ;\IF_2) 
	\]
	that by proposition \ref{prop: Borel_Hirzebruch G2/SO(4)} can be seen as $ \IZ_2 \twoheadrightarrow \free{u_3}$. Hence, 
	$w_3(\rho_2) = u_3$. Considering the definition of the Stiefel-Whitney classes, this means that
	\[ 
		\phi^{-1} \circ Sq^3( th ) \neq 0,  
	\]
	where $th$ is the Thom class associated to $\rho_2$. Then, since 
	\[
		Sq^3 = Sq^1 \circ Sq^2,
	\] 
	$Sq^2(th)  \neq 0$ and  consequently $w_2 \neq 0$. Is not hard to see that $w_2(\rho_2) = u_2$.
\end{proof}

Considering the Gysin exact sequence associated to $\rho_1$, by proposition 
\ref{prop: Borel_Hirzebruch G2/SO(4)} we get that $H^3( G_2/U(2)_{+}; \IF_2 )$ is also trivial. Then, using the same arguments in lemma \ref{lemma: SW rho2}, $\rho_1$ has the same Stiefel-Whitney classes than $\rho_2$. 

\medskip 
Now that we know the Stiefel-Whitney classes of $\rho_1$, we can calculate the cohomology of $G_2/U(2)_{+}$ with $\IF_2$ coefficients as follows: Consider the Serre spectral sequence associated to $\rho_{1}$, with $E_2$-term given by
\[ 
	E_2^{p,q} = H^p \big( G_2/SO(4) ;\mathcal{H}^q( S^2 ;\IF_2 ) \big). 
\]
For an illustration of the $E_2$-term see Figure \ref{Figure: sseq 3}. Notice that, for the same reason exposed in \ref{prop: Cohom G2/U2-}, there are no local coefficients. Then, applying the Gysin sequence to $\rho_1$ \cite[Example 5.C]{mccleary2001user}, and by lemma \ref{lemma: SW rho2}, we get that
\[
	d_3^{0,2}(z) = w_3( \rho_{1} ) = u_3,
\]
where $H^2(S^2; \IF_2) = \free{z}$. Finally, using the Leibniz rule, $d_3^{0,2}$ determinen the complete cohomology of 
$G_2/U(2)_{+}$. We conclude then the following result.

\begin{figure}[h]
	\centering
	\includegraphics[scale=1.1]{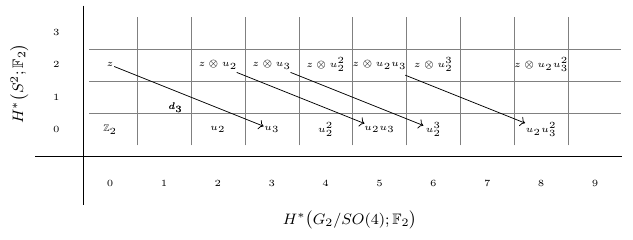}
	\caption{ $E_2^{p,q} = H^p \big( G_2/SO(4) ;H^q( S^2 ;\IF_2 ) \big) \Rightarrow H^*( G_2/U(2)_{+};\IF_2 ) $.}
	\label{Figure: sseq 3}
\end{figure}

\begin{corollary}
\label{corol: Cohom isomorphic}
	The $G_2$ flag manifolds $G_2/U(2)_{+}$ and $G_2/U(2)_{-}$ have isomorphic cohomology rings when we consider $\IF_2$ 	
	coefficients. Moreover, the result in lemma \ref{lemma: Induced map rho_2} applies in the same way to ${\rho_1}^*$.
\end{corollary}

From now on, since $\rho_1$ and $\rho_2$ have the same topological properties considering $\IF_2$ coefficients, we will not make any distinction between them and we will write $G_2/U(2)_{\pm}$ and $\rho_j$. As shown in \cite{kotschickthung}, the similarity mentioned in corollary \ref{corol: Cohom isomorphic} does not happens if we consider the cohomology rings with integer coefficients. 

\medskip
Consider the action of $\IZ_2$ on $G_{2}/U(2)_{\pm}$ by complex conjugation. Since $\rho_{j}(W) = \rho_{j}(\bar{W})$ for every coassociative subspace $W$, and $j \in \{1,2\}$, then both maps are $\IZ_2$-equivariant, with $\IZ_2$  acting trivial on 
$G_2/SO(4)$. This means that we can ask for the Fadell-Husseini index of such bundles. Before we prove our first main result, we fix the notation for the cohomology of the group $\IZ_2$ as
\[
	H^*(B\IZ_2;\IF_2) = \IF_2[t],
\]
with $\deg(t) = 1$.

\begin{theorem} 
\label{thm: BorelCohom G2/U2 & FH rho1&2}
	The Fadell-Husseini index of $\rho_1$ and $\rho_2$ with respect to the introduced $\IZ_2$ action is given by
	\[
		\Index{ \rho_1 }{G_2/SO(4)}{\IZ_2}{\IF_2} =	\Index{ \rho_2 }{G_2/SO(4)}{\IZ_2}{\IF_2} =	\free{  t^3 + u_2 t + u_3  }.
	\]
	Consequently, the Borel cohomology of $G_{2}/U(2)_{\pm} $ is given by 
	\[ 
		H^*( \CBorel{\IZ_2}{G_2/U(2)_{\pm}} ; \IF_2 ) = 	H^*( B\IZ_2 \times  G_2/SO(4) ; \IF_2 )  \Big/ \free{t^3 + u_2 t + u_3}. 
	\] 
\end{theorem}
\begin{proof}
	Consider the Borel construction of the bundle $\rho_i$, for $i \in\{1,2\}$, 
	\[
		\FiberBundle{S^2}{\CBorel{\IZ_2}{G_2/U(2)_{\pm}}}{\rho_i}{\CBorel{\IZ_2}{G_2/SO(4)},}
	\]	
	where the base space can be written as $ B\IZ_2 \times  G_2/SO(4)$. Since the fundamental group
	\[
		\pi_1( B\IZ_2 \times  G_2/SO(4) ) \cong \pi_1(B\IZ_2) \times \pi_1( G_2/SO(4) ) \cong \pi_1(B\IZ_2) \cong \IZ_2
	\]
	acts trivially on the cohomology of the fiber, then the $E_2$-terms of the associated Serre spectral sequence looks as follows:
	\[ 
		E_2^{p,q} = H^p \big( B\IZ_2 \times  G_2/SO(4); H^q( S^2 ;\IF_2) \big). 
	\] 
	Using proposition \ref{prop: Dold} and lemma \ref{lemma: SW rho2}, we get that the transgression map 
	$d_3^{0,2}$ is given by 
	\[ 
		d_3^{0,2}( z ) = t^3 + u_2 t + u_3, 
	\]
	where $H^2( S^2; \IF_2 ) = \free{z}$. Finally, by the Leibniz rule, $d_3^{0,2}$ determines the complete cohomology of 
	$\CBorel{\IZ_2}{G_2/U(2)_{\pm}}$ and the Fadell-Husseini index of $\rho_1$ and $\rho_2$ with coefficients in $\IF_2$.
\end{proof}

\subsection{The Fadell-Husseini index of the pullback bundle $\zeta_n$}
\label{subsec: FadellHusseini pullback}

Let us consider the sphere bundle associated the tautological bundle $\gamma$ over $\OGr{4}{7}$,
\[
	s\gamma = \big( E(s\gamma), \;  \OGr{4}{7},  \; 
			E(s\gamma) \xrightarrow{s\pi} \OGr{4}{7}, \;  S^3 \big),
\] 
and the map $f = c \circ i$, where $i$ denotes the embedding of $G_2/SO(4)$ on $\OGr{3}{7}$ and 
$c \colon \OGr{3}{7} \to \OGr{4}{7}$ assign to every linear subspace its orthogonal complement.
Then, using the map $f$ and the fiber bundle $\rho_j$ defined in \ref{subsec: G2/U2+ } and \ref{subsec: G2/U2- }, 
we construct the following commutative diagram: 
\[
	\xymatrix{
		\calS_{\gamma}^{1} \ar[r] \ar[d]_{\phi_1}&	E\bigC{f^*( s\gamma )}	\ar[r] \ar[d]&		
				E(S\gamma) \ar[d]^{s\pi}\\
		G_2/U(2)_{\pm} \ar[r]^{\rho_{j}} &		G_2/SO(4) \ar[r]^{f} &									\OGr{4}{7},}
\]   
where the map on the left is the pullback bundle of $s\gamma $ along the map $ f \circ \rho_j$,
\[ 
	\zeta_1 = \big( E(\zeta_1) = \calS_{\gamma}^{1}, \;  G_2/U(2)_{-},  \; 
			\calS_{\gamma}^{1} \xrightarrow{\phi_1} G_2/U(2)_{-}, \;  S^3 \big).
\]
To be more precise, the total space of $\zeta_1$ is given by 
\[
	\calS_{\gamma}^{1}  = \{ (W ; \xi^{\perp} , v) \mid W \in G_2/U(2)_{\pm} ,\; \rho_j(W) = \xi ,\;
	 													  v \in  \xi^{\perp} \text{ and }  \deg(v_1)=1\}.
\]

\medskip
The construction of the bundle $\zeta_1$ can be easily generalized as follows: Consider the the $n$-fold product bundle 
$(s\gamma)^n$, and the pullback along the diagonal map $\Delta_n: \OGr{4}{7} \to \OGr{4}{7}^n$:
\[
	\xymatrix{
	E\bigC{ \Delta_n^*( {s\gamma}^n ) } \ar[r] \ar[d]&	E\bigC{s\gamma}^n \ar[d]^{(s\pi)^n} \\
	\OGr{4}{7}	\ar[r]^{\Delta_n} &		\OGr{4}{7}^n.}
\]   
Applying the last construction to the fiber bundle $ \Delta_n^*( {s\gamma}^n ) $, we get a new bundle 
\[ 
	\zeta_n = \big( E(\zeta_n) = \calS_{\gamma}^{n}, \;  G_2/U(2)_{\pm},  \; 
			\calS_{\gamma}^{n} \xrightarrow{\phi_n} G_2/U(2)_{\pm}, \;  (S^3)^n \big).
\]
where 
\begin{eqnarray*}
	\calS_{\gamma}^{n}  & = & \{ (W ; \xi^{\perp} , v_1,  \ldots , v_n) \mid W \in G_2/U(2)_{\pm} ,\; \rho_j(W) = \xi , \\
	 												  &   	       &  \hspace{2cm} v_1, \ldots , v_n \in  \xi^{\perp} \text{ and }  
	 												  \deg(v_1)= \cdots = \deg(v_n)=1\}.
\end{eqnarray*}

\medskip
To describe the Fadell-Husseini index of $\calS_{\gamma}^{n}$, the first step is to calculate the Stiefel-Whitney classes of 
$\zeta_1$. Notice that since $H^1(G_2/U(2)_{\pm} ;\IF_2) = H^3(G_2/U(2)_{\pm} ;\IF_2)=0$, we are just looking for 
$w_2(\zeta_1)$ and $w_4(\zeta_1)$. Then we have the following result.

\begin{lemma} 
\label{lemma: w4}
	The Stiefel-Whitney classes of the pullback bundle $\zeta_1$ are non-zero only in the dimensions $0,2$ and $4$. Moreover, 
	$w_2( \zeta_1)= y$ and $w_4( \zeta_1 )= y^2$.
\end{lemma} 
\begin{proof}
	Let $p \colon \OGr{4}{7} \to \Gr{4}{7}$ be the map that forgets the orientation of the $4$-plane in $\IR^7$. 
	It is well known that $p$ is a $2$-fold covering map, and that the induced homomorphism in cohomology satisfies that
	\[
		p^*(w_i) = w_i(\tilde{\gamma}) = \tilde{w}_i,
	\] 
	where $\tilde{w}_i$ are Stiefel-Whitney classes of $\OGr{4}{7}$. In this way, $\text{im}(p^*)$ is a 
	subalgebra of $H^*(\OGr{4}{7};\IF_2)$ generated by $\{ \tilde{w}_2, \tilde{w}_3, \tilde{w}_4 \}$, except by $\tilde{w}_1$ which 
	is trivial since $\OGr{4}{7}$ is simply connected. 
	
	\medskip
	By the naturality of the Stiefel-Whitney classes it is enough to describe the image of $\tilde{w}_2$ and $\tilde{w}_4$ 
	along the composition 
	\[ 
		G_2/U(2)_{\pm} \xrightarrow{\rho_j} G_2/SO(4)  \xrightarrow{f} \OGr{4}{7}. 
	\]
	Akbulut \& Kalafat in \cite[Theorem 2.11]{akbulutkalafat} proved that the embedding $f:G_2/SO_4\to \OGr{4}{7}$ maps the 	
	euler class of $\gamma$, which reduces to $\tilde{w}_4$, to the integral generator in dimension $4$, 
	which reduces to $u_2^2$. Notice that, since $\tilde{w}_1 = 0$, $\tilde{w}_4 = {\tilde{w}_2}^2$. 
	In this way, $f^{*}$ maps $\tilde{w}_2$ to $u_2$.
	
	\medskip
	On the other hand, by lemma \ref{lemma: Induced map rho_2}, $\rho_j^*(u_2) = y$ and $\rho_j^*(u_2^2) = y^2$. This means 	
	that 	
	\[
		w_2( \zeta_1) = ( f \circ \rho_j )^* ( \tilde{w}_2 ) = y
	\]
	and 
	\[
		w_4( \zeta_1) = ( f \circ \rho_j )^* ( \tilde{w}_4 ) = y^2,
	\]
	for $j \in \{1,2\}$.
\end{proof}

\medskip 
There is an action of $\IZ_2^{n}$ on $\calS^{n}_{\gamma}$ defined as follows: 
\[
	(\beta_1, \ldots ,\beta_n)  \cdot (W; \xi^{\perp}, v_1, \ldots ,v_n ) 
		= (W ; \xi^{\perp} , (-1)^{\beta_1} v_1, \ldots,  (-1)^{\beta_n} v_n ),
\]
for $(\beta_1, \ldots ,\beta_n) \in \IZ_2^{n}$ and $(W; \xi^{\perp}, v_1, \ldots ,v_n ) \in \calS^{n}_{\gamma}$. 
Notice that the action of $\IZ_2^{n}$ on $G_2/U(2)_{\pm}$ is trivial. Also, with this action, the projection map $\phi_n$ is 
$\IZ_2^{n}$-equivariant and we can ask for its Fadell-Husseini index. We will start then with the case $n=1$.

\medskip
For the cohomology of $(S^3)^{n}$ and $B\IZ_2^{n}$ we fix the following notation:
\begin{eqnarray*}
	H^*\big( (S^3)^{n};\IF_2 \big) &\cong & H^*( S^3 ;\IF_2 )^{ \otimes n } \\
										 			   &\cong & \IF_2[z_1]/\free{z_1^2} \otimes \cdots \otimes \IF_2[z_n]/\free{z_n^2} \\
													   &\cong & \IF_2[z_1, \ldots , z_n]/\free{z_1^2, \ldots , z_n^2},
\end{eqnarray*}
where $\deg(z_1) = \cdots = \deg(z_n) = 1$, and 
\[
	H^*( B\IZ_2^{n} ;\IF_2 ) = \IF_2[ t_1, \ldots , t_{n} ],
\]
where $\deg(t_1) = \ldots = \deg(t_{n}) =1$.

\begin{proposition} 
\label{prop: pullback case1}
	Consider the action of $\IZ_2$ on $\calS_{\gamma}^{1}$ introduced before. Then the Fadell-Husseini index of 
	$\Map{\phi_1}{ \calS_{\gamma}^{1} }{ G_2/U(2)_{\pm} } $ is given by
	\[
		\Index{ \phi_1 }{G_2/U(2)_{\pm}}{\IZ_2}{\IF_2} = \free{ y^2 + y t_1^2 + t_1^4 }.
	\]
	Consequently, the Borel cohomology of $S_{\gamma}^{1}$ is described as follows:
	\[ 	
		H^*(  \CBorel{\IZ_2}{S_{\gamma}^{1}}; \IF_2 ) =  
				H^*( \CBorel{\IZ_2}{ G_2/U(2)_{\pm}}; \IF_2 ) \Big/ \free{y^2 + y t_1^2  + t_1^4}.	
	\]
\end{proposition}
\begin{proof} 
	Consider the Serre spectral sequence associated to the bundle 
	\[
		S^3 \hookrightarrow \CBorel{\IZ_2}{\calS_{\gamma}^{1}} \xrightarrow{ id \times_{\IZ_2} \phi_1 }  
		\CBorel{\IZ_2}{G_2/U(2)_{\pm}}, 
	\]
	with $E_2$-term given by 
	\begin{eqnarray*}
		E_2^{p,q} &\cong & H^p\big( E\IZ_2 \times_{\IZ_2} G_2/U(2)_{\pm}  ; \mathcal{H}^q( S^3 ;\IF_2) \big) \\
						 &\cong & H^p\big( B\IZ_2 \times G_2/U(2)_{\pm}  ; \mathcal{H}^q( S^3 ;\IF_2) \big).
	\end{eqnarray*}
	Since the fundamental group $\pi_1( B\IZ_2 \times G_2/U(2)_{\pm}) \cong \IZ_2 $ acts trivially on the cohomology of the 
	fiber, then the $E_2$-terms simplify and looks as follows:
	\[ 
		E_2^{p,q} = H^p\big( B\IZ_2 \times G_2/U(2)_{\pm} ; \IF_2\big) \otimes  H^q\big( S^3 ;\IF_2) \big).
	\] 	

	\medskip
	Notice that the first differential of the spectral sequence associated to $id \times_{\IZ_2} \phi_1$ appears on the $E_4$-term.
	Then, by proposition \ref{prop: Dold} and lemma \ref{lemma: w4}, it is given by 
	\[
		d_4^{0,3}( \xi ) = y^2 + y t_1^2 + t_1^4. 
	\]
	For an illustration of the $E_4$-term see Figure \ref{Figure: sseq 5}.	 This describes the Fadell-Husseini index of 
	$\Map{\phi_1}{ \calS_{\gamma}^{1} }{ G_2/U(2)_{\pm} }$, and by the Leibniz rule, also the cohomology of 
	$ E\IZ_2 \times_{\IZ_2} S_{\gamma}^{1}$. 
		
	\begin{figure}[h]
		\centering
		\includegraphics[scale=.8]{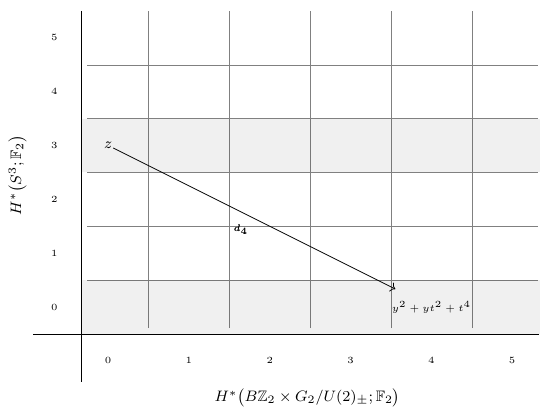}
		\caption{ $E_4^{p,q} = E_2^{p,q} = H^p \big(  B\IZ_2 \times G_2/U(2)_{\pm}  ;H^q( S^3 ;\IF_2) \big)$}
		\label{Figure: sseq 5}
	\end{figure}	
\end{proof}

Using the projection of $\calS_{\gamma}^{n}$ on $\calS_{\gamma}^{1}$, and proposition 
\ref{prop: pullback case1}, we can prove now our second main result.

\begin{theorem} 
\label{thm: FH Pullback}
	Consider the action of $\IZ_2^{n}$ on $\calS_{\gamma}^{n}$  introduced before. Then the
	 Fadell-Husseini index of $\Map{\phi_{n}}{ \calS_{\gamma}^{n} }{ G_2/U(2)_{\pm} } $ is given by
	\[
		\Index{ \phi_{n} }{G_2/U(2)_{\pm}}{\IZ_2^{n}}{\IF_2} 
			= \free{ y^2 + y t_1^2 + t_1^4, \ldots , y^2 + y t_{n}^2 + t_{n}^4 }.
	\]
\end{theorem}
\begin{proof}
	Let $\Map{i_k}{\IZ_2}{\IZ_2^{n}}$ be the inclusion into the $k$th summand, with $1 \leq k \leq n$. 
	The map $i_k$ induces a morphism  between Borel constructions 
	\[
		\xymatrix@C=15mm{
		\CBorel{\IZ_2^{n}}{ \calS_{\gamma}^{1} }  \ar[d]_{id \times_{\IZ_2^{n}} \phi_1 } & 
			\CBorel{\IZ_2}{ \calS_{\gamma}^{1} }  \ar[l]^{ } \ar[d]^{ id \times_{\IZ_2} \phi_1 } \\
		B{\IZ_2}^n \times G_2/U(2)_{\pm} & B\IZ_2 \times G_2/U(2)_{\pm}  \ar[l]^{ }
		}
	\]	
	where $\IZ_2^{n}$ acts on $\calS_{\gamma}^{1} $ as follows: The $k$th summand acts antipodally on the unitary element. 
	We will refer to this action as the one induced by $i_k$. This morphism of Borel constructions induces a morphism between 
	the corresponding Serre spectral sequences, which on the zero column of the $E_2$-term is an isomorphism. 
	Then, by proposition \ref{prop: pullback case1} and the commutativity of the differentials with the morphism between spectral 
	sequences, the transgression map $d_4^{0,3}$ of the spectral sequence associated to 
	$id \times_{\IZ_2^{n}} \phi_1$ is given by
	\[ 
		d_4^{0,3}( z ) =  y^2 + y t_{k}^2 + t_{k}^4,  
	\]
	for every $1 \leq k \leq n$.
	
	\medskip
	Consider now the projection $\Map{p_{k}}{\calS_{\gamma}^{n}}{ \calS_{\gamma}^{1} }$ given by 
	\[
		p_{k}( W; \xi^{\perp}, v_1, \ldots v_n ) = ( W; \xi^{\perp}, v_{k}),
	\]
	for $1 \leq k \leq n$. With respect to the $\IZ_2^{n}$-action on $\calS_{\gamma}^{1}$ induced by $i_k$, the map $p_{k}$ 
	induces a morphism of Borel constructions
	\[
		\xymatrix@C=15mm{
		\CBorel{\IZ_2^{n}}{ \calS_{\gamma}^{n} } \ar[r]^{ } 
			\ar[d]_{ id \times_{\IZ_2^{n}} \phi_n } & \CBorel{\IZ_2^{n}}{ \calS_{\gamma}^{1} } 
			\ar[d]^{ id \times_{\IZ_2^{n}} \phi_1 } \\
		B{\IZ_2}^n \times G_2/U(2)_{\pm} \ar[r]^{ } & B{\IZ_2}^n \times G_2/U(2)_{\pm}. 
		}
	\]
	This morphism of Borel constructions, in turn, induces a morphism between the corresponding Serre spectral sequences, 
	which on the zero column of the $E_2$-term is a monomorphism. For an illustration of the morphism between the 
	corresponding Serre spectral sequences see Figure \ref{Figure: sseq 6}. Then, by the commutativity of the differentials, the 
	map $d_4^{0,3}$ of the spectral sequences associated to $id \times_{\IZ_2^{n}} \phi_n$ is given by 
	 \[ 
	 	d_4^{0,3}( z_{k} ) = d_4^{0,3}( p_{k}^*(z) ) =\id ( d_4^{0,3}( z ) ) = y^2 + y t_{k}^2 + t_{k}^4, 
	 \]
	 for every generator $z_{k}$ in $H^3( (S^3)^n;\IF_2 )$, with $1 \leq k \leq n$. Then, by the Leibnitz rule, this describes 
	 the Fadell-Husseini index of $\Map{\phi_n}{ \calS_{\gamma^{\perp}}^{n} }{ G_2/U(2)_{\pm} } $.
\end{proof}

\begin{figure}[H] 
		\centering
		\includegraphics[scale=.65]{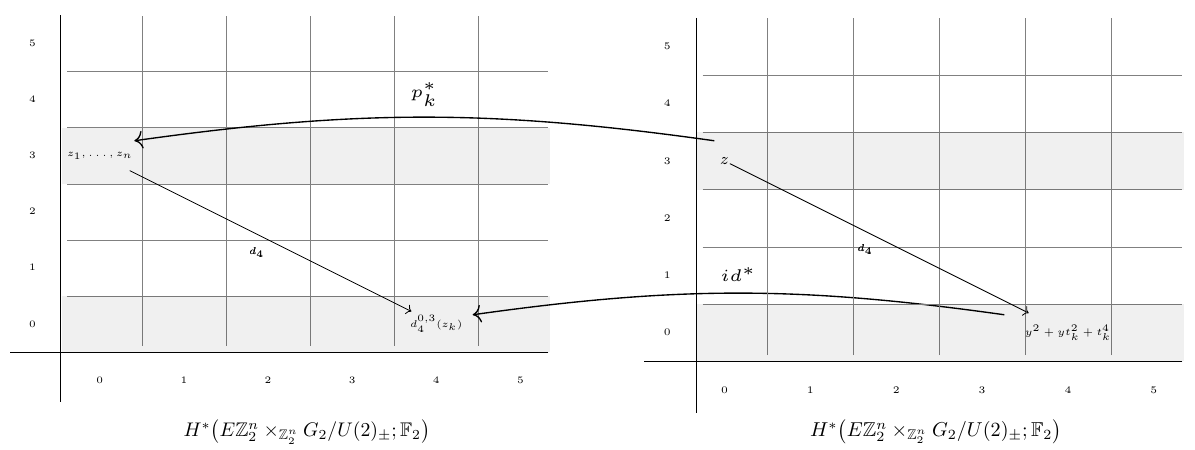}
		\caption{Morphism of spectral sequences induced by $p_k$.}
		\label{Figure: sseq 6}
\end{figure}


\section{ \Large Application: Associative subspaces in Discrete Geometry.}
\label{sec: Application}

Ones we studied the associative and coassociative subspaces of $\text{Im}(\mathbb{O})$, as well as the Fadell-Husseini index of the fibrations defined in the previous sections, in this section we will show how this kind of subspaces appears within a classical problem in discrete geometry known as \emph{ the Gr\"unbaum mass partition problem} \cite{Grunbaum-1960}.

\medskip
From now on, a mass is assumed to be a finite Borel measure on a Euclidean space which vanishes on each affine hyperplane.
The general problem considers a mass $\mu$ in $\IR^n$, and looks for a collection of oriented affine hyperplanes $H_1, H_2, \ldots , H_n$ in $\IR^n$, also called $n$-arrangement, such that each of the $2^n$ resulting orthants contains the same amount of the mass $\mu$, i.e., 
\[
	\mu(H_1^{\sigma_1} \cap \cdots \cap H_n^{\sigma_n}) = \frac{1}{2^n} \mu(\IR^n)
\]
for every $(\sigma_1, \ldots , \sigma_n) \in \IZ_2^n$, where $H_i^{\sigma_i}$ denotes the appropiate closed halfspace deterinated by $H_i$. In that case we said that the collection of hiperplanes equiparts the mass $\mu$.

\medskip
While the case $n=2$ is the well-known ham sandwich theorem, the positive answer for the case $n=3$ was given by Hadwiger in \cite{Hadwiger-1966}. On the other hand, in \cite{Avis-1984}, Avis showed that in every dimension $n \geq 5$ there is a mass which cannot be equiparted by an $n$-arragement. The problem becomes more complicated in the case $n=4$, which, to this very day, is still open. The best approach to the solution was given by Dimitrijevi{\'c} Blagojevi{\'c} in \cite{blagojevic2009mass}, where she considers an almost equiparticion. To be more precise, given a $4$-arrangement  $\{H_1, H_2, H_3, H_4\}$ in 
$\mathbb{R}^4$, and consequently sixteen orthants $\textbf{O}_{\sigma_1 \sigma_2 \sigma_3 \sigma_4} = 
H_1^{\sigma_1} \cap H_2^{\sigma_2} \cap H_3^{\sigma_3} \cap H_4^{\sigma_4}$, an almost equipartition means that
\begin{eqnarray*}
	\mu( \textbf{O}_{0000} ) &=&  \mu( \textbf{O}_{0010} ) =  \mu( \textbf{O}_{0101} ) =  \mu( \textbf{O}_{0111} ) \\
											   &=& \mu( \textbf{O}_{1000} ) = \mu( \textbf{O}_{1010} ) =  \mu( \textbf{O}_{1101} )
													=  \mu( \textbf{O}_{1111} )
\end{eqnarray*}
\begin{eqnarray*}
	\mu( \textbf{O}_{0001} ) &=&  \mu( \textbf{O}_{0011} ) =  \mu( \textbf{O}_{0100} ) =  \mu( \textbf{O}_{0110} ) \\
											   &=& \mu( \textbf{O}_{1001} ) = \mu( \textbf{O}_{1011} ) =  \mu( \textbf{O}_{1100} )
													=  \mu( \textbf{O}_{1110} ).
\end{eqnarray*}

For more details on the history and discuccion of solution methods regarding   Gr\"unbaum's mass partition problem consult \cite{Blagojevic-Frick-Haase-Ziegler-2018}, and for the currently best known results see \cite{Blagojevic-Frick-Haase-Ziegler-2016}.
 
\medskip
In this section, motivated by the computations presented by Blagojevi{\'c}, Calles Loperena, Crabb \& Dimitrijevi{\'c} Blagojevi{\'c} in \cite{BlagojevicCalles2023}, we will show that the solutions provided by Dimitrijevi{\'c} Blagojevi{\'c} in \cite{blagojevic2009mass} can be consider, in a sense, as the orthogonal complement of an associative subspace in $\text{Im}(\mathbb{O})$. More precisely, we prove the following result

\begin{theorem}
\label{thm:  Application}
	Let $\mu$ be a mass assignment on the Grassmann manifold $\OGr{4}{7}$. Then there exist a vector subspace 
	$L \in \OGr{4}{7}$, which is the orthogonal complement of an associative subspace in $\text{Im}(\mathbb{O})$, and a 
	$4$-arrangement 	$(H_1, H_2, H_3, H_4)$ in $L$ which almost equipart the mass $\mu^L$.
\end{theorem}

Here we define a mass assignment on a Grassmann manifold as a cross section of the fibre bundle 
\begin{equation}  \label{eq: mass_assignment}	
	\xymatrix{
		M_+'(\IR^{\ell})\ar[r] & \mathcal{M}_+'(\ell,d)\ar[r]^-{\rho} & \OGr{\ell}{d},
	}
\end{equation}
where $M_+'(\IR^{\ell})$ denotes the subspace of all masses on $\IR^{\ell}$, the total space is given by
\[
	\mathcal{M}_+'(\ell,d):= \{ (L, \nu) \mid L \in \OGr{\ell}{d} \; \text{and} \; \nu \in M_+'(L) \},
\]
and the map $\rho$ is give by $( L, \nu )\longmapsto L$. This means that $\mu$ assigns to each subspace $L$ a mass $\mu^L$ in $L$. For more details about mass assignments see \cite{Schnider2020}, \cite{axelrod2022bisections} and \cite{BlagojevicCalles2023}.

\subsection{ Proof (theorem \ref{thm:  Application}) }
\label{subsec: proof}

The proof is based on the \emph{configuration space test map scheme (CS/TM - scheme)}, developed in numerous paper and formalized by \v{Z}ivaljevi\'{c} in \cite{Zivaljevic-1996, vzivaljevic1998}.
The main idea is to reduce the problem to the question about the non-existence of a particular equivariant map. 
We will consider a map $f \colon \mathcal{C} \to Z$, where $\mathcal{C}$ is the space of all candidates to be a solution of our 
problem, also called 	\emph{the configuration space}, $Z$ is called \emph{the test space}, $f$ is called \emph{the test map}, 
and there is a subspace $V \subseteq Z$ such that $x \in \mathcal{C}$ is a solution of the problem if and only if $f(x) \in V$. 
I addition there is a group $G$, of symmetries of the problem, which acts on $\mathcal{C}$ and $Z$, keeping the subspace 
$V$ $G$-invariant. Moreover, the test map $f$ is equivariant.  In this way, if we prove that 
$f \colon \mathcal{C} \to Z \setminus V$ does not existe, we guarantee that there is a solution of our problem. 

\medskip
In order to apply topological methods and derive an appropiate configuration test map scheme for our problem we make an 
additional assumption on the mass assignment. We assume that for every linear subspace $L \in \OGr{4}{7}$ the associated mass 
$\mu^{L}$ has compact and connected support. This means that for every direction in $L$ there exists a unique oriented affine 
hyperplane orthogonal to the direction equiparting $\mu^{L}$.
	
\subsubsection{Configuration space}
\label{subsubsection: ConfigSpace}
To construct the equivariant map of the test map scheme, 
let us start considering a coassociative subspace $W \in G_2/U(2)_{\pm}$ such that $\rho_i(W) = \xi$, with $\xi \in G_2/SO(4)$ an associative subspace of $\text{Im}(\mathbb{O})$.
Since $W$ corresponds to an orthogonal almost complex structure $J$ of $\xi^{\perp}$, given two unit vectors 
$v_1, v_2 \in \xi^{\perp}$, we obtain a collection $\{v_1, J(v_1), v_2, J(v_2)\}$, which by the extra assumption over the mass 
assignment $\mu$, induces a $4$-arrangement of oriented affine hyperplanes $\{ H_{v_1}, H_{J(v_1)}, H_{v_2}, H_{J(v_2)} \}$ 
equiparting the mass $\mu^{\xi^{\perp}}$ in $\xi^{\perp}$. In that way, $\calS_{\gamma}^{2}$, the total space of the pullback bundle $\phi_2$ defined in section \ref{subsec: FadellHusseini pullback}, represent the space of all cantidates to be a solutions of our problem.

\subsubsection{Test map}
\label{subsubsection: TestMap}
Consider now the bundle map 
$f \colon \calS_{\gamma}^{2} \to G_2/U(2)_{\pm} \times \mathbb{R}^{16}$, between $\phi_{2}$ and the projection into the 
first 	factor $\pi$, given by 
\begin{align*}
	(W ; \xi^{\perp}, v_1, v_2)	\longmapsto
	  & 
	\Big(W ;
		\big(\mu^{\xi^{\perp}}(
			H_{v_1}^{\sigma_1} \cap H_{J(v_1)}^{\sigma_2} \cap H_{v_2}^{\sigma_3} \cap H_{J(v_2)}^{\sigma_4}) 
				- \tfrac{1}{16}\mu^{\xi^{\perp}}(\xi^{\perp}) \big)_{(\sigma_1,\sigma_2,\sigma_3,\sigma_4)\in \IZ_2^4}  
	\Big),
\end{align*}
where $\mathbb{R}^{16}$ is indexed by the elements of the group $\IZ_2^4$. Notice that, by the extra assumption over the 
mass assignment $\mu$, the codomain of the map $f$ is a subspace $\mathcal{U}$ of $\mathbb{R}^{16}$ defined by the 
following equalities: 
\begin{eqnarray*}
	\sum x_{0 \sigma_1 \sigma_2 \sigma_3} &=& \sum x_{1 \sigma_1 \sigma_2 \sigma_3};
		\hspace{1cm} \sum x_{\sigma_1 0 \sigma_2 \sigma_3} = \sum x_{\sigma_1 1 \sigma_2 \sigma_3} \\
	\sum x_{\sigma_1 \sigma_2 0 \sigma_3} &=& \sum x_{\sigma_1 \sigma_2 1 \sigma_3};
		\hspace{1cm}  \sum x_{\sigma_1 \sigma_2 \sigma_3 0} = \sum x_{\sigma_1 \sigma_2 \sigma_3 1} \\
	\sum x_{\sigma_1 \sigma_2 \sigma_3 \sigma_4} &=& 0, 
\end{eqnarray*}
for $(\sigma_1, \sigma_2, \sigma_3) \in \IZ_2^3$ and $(\sigma_1, \sigma_2, \sigma_3, \sigma_4) \in \IZ_2^4$. 

\medskip
As shown in section \ref{subsec: FadellHusseini pullback}, the group $\IZ_2^2$ acts antipodally on the unit vectors of 
$\calS_{\gamma}^{2}$. 
There is also an action of $\IZ_2^2$ on $\mathbb{R}^{16}$ given by permuting the elements in the following way: 
\[
(\beta_1, \beta_2) \cdot x_{\sigma_1 \sigma_2 \sigma_3 \sigma_4} = 
	x_{\beta_1 +\sigma_1 \; \beta_1 + \sigma_2 \; \beta_2 +  \sigma_3 \; \beta_2 +\sigma_4},
\]
where $(\beta_1, \beta_2) \in \IZ_2^2$ and $x_{\sigma_1 \sigma_2 \sigma_3 \sigma_4} \in \mathbb{R}^{16}$. All this makes 
the bundle map $f$ $\IZ_2^2$-equivariant over the codomain $G_2/U(2)_{\pm} \times \mathcal{U}$.

\subsubsection{Test space} 
\label{subsubsection: TestSpace}

Let us consider the $\IZ_2^2$-invariant subspace  $L$ of $\mathbb{R}^{16}$ generated by the equalities
\begin{eqnarray*}
	x_{0000} =  x_{0010} = x_{0100} =  x_{0110} = x_{1001} =  x_{1011} = x_{1101} =  x_{1111}, \\
	x_{0001} =  x_{0011} = x_{0101} =  x_{0111} = x_{1000} =  x_{1010} = x_{1100} =  x_{1110}.
\end{eqnarray*}
Notice that the previous equations determine the condition of ``almost equipartition'' of our $4$-arrangement.
In this way, the test space of our problem is the subspace $\mathcal{U} \cap L$ of codimension $10$ inside $\mathcal{U}$.

\medskip
All of this leads us to the following result, which relates our mass partition problem with a problem about the existence of an equivariant map (the main idea of the CS/TM scheme).

\begin{proposition}
	If there is no $\IZ_2^2$-equivariant map 
	\[
		\calS_{\gamma}^{2} \rightarrow G_2/U(2)_{\pm} \times (\mathbb{U} \setminus L), 
	\]
	with the already defined actions, then our problem stated in Theorem \ref{thm:  Application} has a solution.
\end{proposition}

To make the upcoming computations more manageable, we consider a $\IZ_2^2$ deformation retraction between 
$\mathbb{U} \setminus L$ and $S( (\mathbb{U} \cap L)^{\perp} )$, where $S( (\mathbb{U} \cap L)^{\perp} )$ is a sphere of dimension $9$.

\subsubsection{The non-existance of the equivariant map.} 
\label{subsubsection: Indexes}
For the previously mentioned reasons, the problem reduces to  the non-existence of the bundle map 
\[
	f \colon \calS_{\gamma}^{2} \rightarrow G_2/U(2)_{\pm} \times S( (\mathbb{U} \cap L)^{\perp} )
\]
between $\phi_{2}$ and the corresponding sphere bundle $S\pi$. If we suppose that such map exists, by the monotonicty 
of the index 
\[
	\Index{ \phi_{2} }{G_2/U(2)_{\pm}}{\IZ_2^{2}}{\IF_2} \supseteq \Index{ S\pi}{G_2/U(2)_{\pm}}{\IZ_2^{2}}{\IF_2},
\]
which by theorem \ref{thm: FH Pullback}, and the computations made in \cite[Section 3.2]{blagojevic2009mass}, as well as \cite[Proposition 3.1]{fadellhusseini}, simplifies to
\[
	\free{ y^2 + y t_1^2 + t_1^4,  y^2 + y t_{2}^2 + t_{2}^4 }
		\supseteq 
	H^*( G_2/U(2)_{\pm} ;\IF_2 ) \otimes \free{ t_1^{5}t_2^{2} + t_1^{4}t_2^{3} + t_1^{3}t_2^{4} + t_1^{2}t_2^{5} }.
\]
For a more details about the index of a projection see \cite[Section 3.2]{BlagojevicCalles2023}.
Notice that the index of the sphere representation differs from the original result because of the number of masses we are considering, as well as the action of $\IZ_2^2$ on the codomain. 
However, removing the additional information, and assuming that $t_1 = t_2$ and $t_3 = t_4$ on the generator of 
$H^*(B\IZ_2^4;\IF_2)$ in \cite[Section 3.2]{blagojevic2009mass}, we get with the desired result. 

\medskip
Finally, using a computational package like \emph{SINGULAR} or \emph{SAGE} we confirm that
\[
	t_1^{5}t_2^{2} + t_1^{4}t_2^{3} + t_1^{3}t_2^{4} + t_1^{2}t_2^{5}
	\notin
	\free{ y^2 + y t_1^2 + t_1^4,  y^2 + y t_{2}^2 + t_{2}^4 },
\]
which means that the $\IZ_2^2$-equivariant bundle map $f$ does not exist. As a consequence, there exist an associative subspace $\xi$, and a $4$-arrangement $(H_1, H_2, H_3, H_4)$ in $\xi^{\perp}$ which almost equipart the mass 
$\mu^{\xi^{\perp}}$ in $\xi^{\perp}$.


\bibliography{Ref.bib}{}
\bibliographystyle{amsplain}

\end{document}